\def\marker{\>\hbox{${\vcenter{\vbox{
    \hrule height 0.4pt\hbox{\vrule width 0.4pt height 6pt
    \kern6pt\vrule width 0.4pt}\hrule height 0.4pt}}}$}\>}

\documentclass[12pt]{article}
\usepackage{amsmath,amssymb,amsthm}
\usepackage{graphicx}
\usepackage{subfigure}
\usepackage{psfrag}
\usepackage{color}
\usepackage{enumerate}
\usepackage{epstopdf}
\usepackage{indentfirst}
\newtheorem{thm}{Theorem}[section]

\newtheorem{cor}[thm]{Corollary}
\newtheorem{lem}[thm]{Lemma}
\newtheorem{prop}[thm]{Proposition}
\theoremstyle{definition}
\newtheorem{defn}[thm]{Definition}
\def\dfn#1{{\it #1}}

\def\qed{ \hfill $\square$}

\title{A note on rainbow saturation number of paths}

\author{Shujuan Cao \footnotemark[1] ,\quad
Yuede Ma \footnotemark[2] \quad and \quad
Zhenyu Taoqiu\footnotemark[3]\quad
}
\date{\today}

\begin{document}
\maketitle
\begin{abstract}
For a fixed graph $F$ and an integer $t$, the \dfn{rainbow saturation number} of $F$, denoted by $sat_t(n,\mathfrak{R}(F))$, is defined as the minimum number of edges in a $t$-edge-colored graph on $n$ vertices which does not contain a
\dfn{rainbow copy} of $F$, i.e., a copy of $F$ all of whose edges receive a different color, but the addition of any missing edge in any color from $[t]$ creates such a rainbow copy.
Barrus, Ferrara, Vardenbussche and Wenger prove that $sat_t(n,\mathfrak{R}(P_\ell))\ge n-1$ for $\ell\ge 4$ and $sat_t(n,\mathfrak{R}(P_\ell))\le \lceil \frac{n}{\ell-1} \rceil \cdot \binom{\ell-1}{2}$ for $t\ge \binom{\ell-1}{2}$, where $P_\ell$ is a path with $\ell$ edges. In this short note, we improve the upper bounds and show that $sat_t(n,\mathfrak{R}(P_\ell))\le \lceil \frac{n}{\ell} \rceil \cdot \left({{\ell-2}\choose {2}}+4\right)$ for $\ell\ge 5$ and $t\ge 2\ell-5$.\\

\noindent\textbf{Keywords:} rainbow saturation number, edge-coloring, path\\
\end{abstract}

\renewcommand{\thefootnote}{\fnsymbol{footnote}}
\footnotetext[1]{
School of Mathematical Sciences, Tiangong University, Tianjin 300387, China;
sj.cao@163.com. }
\footnotetext[2]{
School of Science, Xi¡¯an Technological University, Xi¡¯an, Shaanxi 710021, PR China;
mayuede0000@163.com. }
\footnotetext[3]{
Center for Combinatorics and LPMC,
Nankai University, Tianjin 300071, China;
tochy@mail.nankai.edu.cn (the corresponding author). }

\section{Introduction}

Throughout this note, all graphs are simple, undirected, and finite. Throughout we use
the terminology and notation of \cite{W}. For a positive integer $t$, let $[t]$ denote the set $\{1,\ldots, t \}$. A \dfn{$t$-edge-coloring} of a graph $G$ is a function $f: E(G)\rightarrow [t]$, and a graph equipped with such a coloring is a ($t$-edge-colored graph).\par

A graph $G$ is called \dfn{$F$-saturated} if it is a maximal $F$-free graph. The classical saturation problem, first studied by Zykov~\cite{Z} and Erd{\"o}s, Hajnal and Moon~\cite{EH}, asks for the minimum number of edges in an $F$-saturated graph. For more results on saturation numbers, the reader should consult the excellent survey of Faudree, Faudree, and Schmitt \cite{FFS2011}. A rainbow analog of this problem was recently introduced by Barrus, Ferrara, Vardenbussche and Wenger~\cite{BF}, where a $t$-edge-colored graph is defined to be \dfn{rainbow $F$-saturated} if it contains no rainbow copy of $F$, i.e., a copy of $F$ all of whose edges receive a different color, but the addition of any missing edge in any color creates such a rainbow copy, denoted by \dfn{$(\mathfrak{R}(F),t)$-saturated}. This minimum size of a $t$-edge-colored rainbow $F$-saturated graph, denoted by $sat_t(n,\mathfrak{R}(F))$, is the \dfn{rainbow saturation number} of $F$, i.e., $$sat_t(n,\mathfrak{R}(F))=\min \{|E(G)| \colon |V(G)|=n, G\ \text{is $(\mathfrak{R}(F),t)$-saturated}\}.$$\par

In \cite{BF}, the authors proved some results on $sat_t(n,\mathfrak{R}(F))$ of various families of graphs including complete graphs, trees and cycles.  Kor{\'a}ndi \cite{DK} prove a conjecture of Barrus et al. on the rainbow saturation number of complete graphs. For more results on this topic we refer to \cite{FJ,GI,AG}. Especially, Barrus et al. \cite{BF} proved the following results on paths.
\begin{thm}[\cite{BF}]{\label{I1}}
(i). $sat_t(n,\mathfrak{R}(P_\ell))\ge n-1, \ \ell\ge 4$.
(ii). $sat_t(n,\mathfrak{R}(P_4))=n-1, \ t\ge 8$.
(iii).
$sat_t(n,\mathfrak{R}(P_\ell))\le \lceil \frac{n}{\ell-1} \rceil \cdot \binom{\ell-1}{2}, \ t\ge \binom{\ell-1}{2}$.
\end{thm}

In the proof of Theorem \ref{I1}(iii), Barrus et al. use rainbow $K_{\ell-1}$ as construction components. Here we improve the upper bounds by changing the construction components, which yields Theorem \ref{thm1}.
\begin{thm}\label{thm1}
For $\ell\ge 5$ and $t\ge 2\ell-5$, $sat_t(n,\mathfrak{R}(P_\ell))\le \lceil \frac{n}{\ell} \rceil \cdot \left(\binom{\ell-2}{2}+4\right)$.
\end{thm}

In Section 2, for all positive integers $\ell\ge 5$, we provide infinitely many $t$-edge-colored graphs with $sat_t(n,\mathfrak{R}(P_\ell))$ attaining the upper bounds. In Section 3, the proof of the main result is presented.

\section{Construction}

In this section, we construct several rainbow $P_\ell$-saturated graphs achieving the upper bounds in Theorem \ref{thm1}, for all positive integers $\ell\ge 5$.

\begin{defn}\label{Cd1}
For any positive integer $\ell\ge 5$, let $V(K_{\ell-2})=\{v_0, v_1, \ldots, v_{\ell-3} \}$ and $c(v_i v_j)=i+j$ as the edge coloring rule of $K_{\ell-2}$.
Let $H$ be the graph obtained from $K_{\ell-2}$ by adding two vertices $\{v_{\ell-2},x\}$ and four edges $v_{\ell-2}v_{\ell-3},v_{\ell-2}v_{\ell-4},x v_{\ell-3},xv_{\ell-4}$. Set $c(v_{\ell-2}v_{\ell-3})=c(xv_{\ell-4})=2\ell-5$ amd $c(v_{\ell-2}v_{\ell-4})=c(xv_{\ell-3})=2\ell-6$. Let $G^*=kH$ for all positive integers $k\ge 2$.
\end{defn}

In the following we will study some properties of $H$ and $G^*$.

\begin{lem}\label{Ct1}
The graph $H$ is $t$-proper-edge-colored, and $H$ does not contain any rainbow copy of $P_\ell$, where $t=2\ell-5$ and $\ell\ge 5$.
\end{lem}

\begin{proof}
By the edge coloring rule of $H$, for any $v\in H$, all edges incident to $v$ have different colors. Thus $H$ is $t$-proper-edge-coloring, and $t=\ell-2+\ell-3=2\ell-5$.
Suppose $H$ contains a rainbow copy of $P_\ell$, denoted by $P_\ell'$. For $|H|=\ell$, we have $x,v_{\ell-2}\in P_\ell'$. Then there are at least two same colors in $c(P_\ell')$.
\end{proof}

For a graph $G$ and $v\in G$, let $P_v$ be a rainbow Hamilton path from $v$ of $G$ and $c(P_v)$ be the color set of $E(P_v)$. Let $G_v$ be the set of all such $P_v$ and $c(G_v)=\bigcap c(P_v)$.

\begin{lem}\label{Ct2}
For any positive integer $\ell\ge 5$, let $H^*=H\setminus x$ and $e\in E(\overline{G^*})$. Then $H^*+e$ contains a rainbow $P_\ell$ such that $c(e)$ is any color from $[t]$ and $e$ is incident to $v_i\in H^*$, where $i\notin \{\ell-4,\ell-3 \}$.
\end{lem}

\begin{proof}
We only need to prove $c(H^*_v)=\emptyset$ for any $v\in H^* \setminus \{v_{\ell-4},v_{\ell-3} \}$.\par
Suppose $\ell$ is odd. Assume $a\in N^+$, then $V(H^*)=\{v_0,v_1,\ldots,v_{\ell-4},v_{\ell-3},v_{\ell-2} \}$. By Definition \ref{Cd1}, $t=2\ell-5$ and $v_{\ell-2}$ is only adjacent to $v_{\ell-3}, \ v_{\ell-4}$.

If $v=v_0$, for $H^*$, then let $P_{v_0}^1={v_0,v_1,v_3,\ldots, v_{2a+1},\ldots, v_{\ell-4},v_{\ell-2},v_{\ell-3},\ldots, v_{2a},\ldots, v_4,v_2}$ and $P_{v_0}^2=\{v_0,v_1,v_2,v_3,\ldots, v_{2a},v_{2a+1},\ldots, v_{\ell-4},v_{\ell-3},v_{\ell-2} \}$. Then all elements in $c(P_{v_0}^1)$ are even except for $\{1,2\ell-5 \}$ and all elements in $c(P_{v_0}^2)$ are odd. So we have $c(H^*_{v_0})\subseteq \{1,2\ell-5\}$. Change the positions of a few vertices in $P_{v_0}^2$, we can get $P_{v_0}^{21}=\{v_0,v_2,v_1,v_3,\ldots, v_{\ell-5},v_{\ell-3},$ $v_{\ell-4},v_{\ell-2} \}$. There is no $\{1,2\ell-5 \}$ in $c(P_{v_0}^{21})$. Thus $c(H^*_{v_0})=\emptyset$.\par

If $v=v_{2a}$, for $H^*$, then let $P_{v_{2a}}^1=\{v_{2a},v_{2a-2},\ldots, v_4,v_2,v_0,v_1,v_3,\ldots, v_{\ell-4},v_{\ell-2},v_{\ell-3},\ldots, v_{2a+2} \}$ and $P_{v_{2a}}^2=\{v_[2a],v_0,v_1,v_2,v_3,\ldots, v_{2a-1},v_{2a+1},\ldots, v_{\ell-4},v_{\ell-3},v_{\ell-2} \}$. Then all elements in $c(P_{v_{2a}}^1)$ are even except for $\{1,2\ell-5 \}$ and all elements in $c(P_{v_{2a}}^2)$ are odd except for $\{2a,4a \}$. So we have $c(H^*_{v_{2a}})\subseteq \{1,2\ell-5,2a,4a \}$. Change the positions of a few vertices in $P_{v_{2a}}^2$, we can get $P_{v_{2a}}^3=\{v_{2a},v_{2a-1},\ldots, v_3,v_1,v_2,v_0,v_{2a+2},v_{2a+1},v_{2a+3},v_{2a+4},\ldots, v_{\ell-5},v_{\ell-3},v_{\ell-4},v_{\ell-2} \}$. There is no $\{1,2a,4a,2\ell-5 \}$ in $c(P_{v_{2a}}^3)$. Thus $c(H^*_{v_{2a}})=\emptyset$.\par

If $v=v_{2a+1}$, for $H^*$,  then let $P_{v_{2a+1}}^1=\{v_{2a+1},v_{2a-1},\ldots, v_3,v_1,v_0,v_2,v_4,\ldots, v_{\ell-3},v_{\ell-2},v_{\ell-4},\ldots,$ $v_{2a+3} \}$ and $P_{v_{2a+1}}^2=\{v_{2a+1},v_1,v_0,v_2,v_3,\ldots, v_{2a},v_{2a+2},\ldots, v_{\ell-4},v_{\ell-3},v_{\ell-2} \}$. Then all elements in $c(P_{v_{2a+1}}^1)$ are even except for $\{1,2\ell-5 \}$ and all elements in $c(P_{v_{2a+1}}^2)$ are odd except for $\{2,2a+2,4a+2 \}$. So we have $c(H^*_{v_{2a+1}})\subseteq \{1,2\ell-5,2,2a+2,4a+2 \}$. Change the positions of a few vertices in $P_{v_{2a+1}}^2$, we can get $P_{v_{2a+1}}^3=\{v_{2a+1},v_{2a},\ldots, v_3,v_1,v_2,v_0,v_{2a+2},v_{2a+3},\ldots, v_{\ell-3},$ $v_{\ell-4},v_{\ell-2} \}$ and $P_{v_{2a+1}}^{21}=\{v_{2a+1},v_3,v_2,v_1,v_0,v_4,\ldots, v_{\ell-2} \}$. There is no $\{1,2,4a+2,2\ell-5 \}$ in $c(P_{v_{2a+1}}^3)$ and no $2a+2$ in $c(P_{v_{2a+1}}^{21})$. Thus $c(H^*_{v_{2a+1}})=\emptyset$.\par

If $v=v_{\ell-5}$, for $H^*$, then let $P_{v_{\ell-5}}^1=\{v_{\ell-5},v_{\ell-7},\ldots, v_4,v_2,v_0,v_1,v_3,\ldots, v_{\ell-4},v_{\ell-2},v_{\ell-3} \}$ and $P_{v_{\ell-5}}^2=\{v_{\ell-5},v_0,v_1,v_2,v_3,\ldots, v_{\ell-6},v_{\ell-4},v_{\ell-3},v_{\ell-2} \}$. Then all elements in $c(P_{v_{\ell-5}}^1)$ are even except for $\{1,2\ell-5 \}$ and all elements in $c(P_{v_{\ell-5}}^2)$ are odd except for $\{\ell-5,2\ell-10 \}$. So we have $c(H^*_{v_{\ell-5}})\subseteq \{1,2\ell-5,\ell-5,2\ell-10 \}$. Change the positions of a few vertices in $P_{v_{\ell-5}}^2$, we can get $P_{v_{\ell-5}}^3=\{v_{\ell-5},v_{\ell-6},\ldots, v_3,v_1,v_2,v_0,v_{\ell-3},v_{\ell-4},v_{\ell-2} \}$. There is no $\{1,2\ell-5,\ell-5,2\ell-10 \}$ in $c(P_{v_{\ell-5}}^3)$. Thus $c(H^*_{v_{\ell-5}})=\emptyset$.\par

If $v=v_{\ell-4}, \ \ell\geq 7$, for $H^*$, then let $P_{v_{\ell-4}}^1=\{v_{\ell-4},v_{\ell-6},\ldots, v_3,v_1,v_0,v_2,v_4,\ldots, v_{\ell-3},v_{\ell-2} \}$ and $P_{v_{\ell-4}}^2=\{v_{\ell-4},v_1,v_0,v_2,v_3,\ldots, v_{\ell-5},v_{\ell-3},v_{\ell-2} \}$. Then all elements in $c(P_{v_{\ell-4}}^1)$ are even except for $\{1,2\ell-5 \}$ and all elements in $c(P_{v_{\ell-4}}^2)$ are odd except for $\{2,\ell-3,2\ell-8 \}$. So we have $c(H^*_{v_{\ell-4}})\subseteq \{1,2\ell-5,2,\ell-3,2\ell-8 \}$. Change the positions of a few vertices in $P_{v_{\ell-4}}^2$, we can get $P_{v_{\ell-4}}^3=\{v_{\ell-4},v_{\ell-2},v_{\ell-3},v_0,v_2,v_1,v_3,\ldots, v_{\ell-5} \}$ and $P_{v_{\ell-4}}^{31}=\{v_{\ell-4},v_{\ell-2},v_{\ell-3},v_2,v_0,v_1,v_3,\ldots,$ $v_{\ell-5} \}$. There is no $\{1,2,2\ell-8 \}$ in $c(P_{v_{\ell-4}}^3)$ and no $\ell-3$ in $c(P_{v_{\ell-4}}^{31})$. Thus $c(H^*_{v_{\ell-4}})=\{2\ell-5 \}$.\par

If $v=v_{\ell-3}, \ \ell\geq 7$, for $H^*$, then let $P_{v_{\ell-3}}^1=\{v_{\ell-3},v_{\ell-5},\ldots, v_4,v_2,v_0,v_1,v_3,\ldots, v_{\ell-4},v_{\ell-2} \}$ and $P_{v_{\ell-3}}^2=\{v_{\ell-3},v_0,v_1,v_2,\ldots, v_{\ell-5},v_{\ell-4},v_{\ell-2} \}$. Then all elements in $c(P_{v_{\ell-3}}^1)$ are even except for $1$ and all elements in $c(P_{v_{\ell-3}}^2)$ are odd except for $\{\ell-3,2\ell-6 \}$. So we have $c(H^*_{v_{\ell-3}})\subseteq \{1,\ell-3,2\ell-6 \}$. Change the positions of a few vertices in $P_{v_{\ell-3}}^2$, we can get $P_{v_{\ell-3}}^{21}=\{v_{\ell-3},v_0,v_2,v_1,v_3,\ldots, v_{\ell-2} \}$ and $P_{v_{\ell-3}}^{22}=\{v_{\ell-3},v_2,v_0,v_1,v_3,v_4,\ldots, v_{\ell-2} \}$. There is no $1$ in $c(P_{v_{\ell-3}}^{21})$ and no $\ell-3$ in $c(P_{v_{\ell-3}}^{22})$. Thus $c(H^*_{v_{\ell-3}})=\{2\ell-6 \}$.\par

If $v=v_{\ell-2}$, for $H^*$, then let $P_{v_{\ell-2}}^1=\{v_{\ell-2},v_{\ell-3},v_{\ell-5},\ldots, v_4,v_2,v_0,v_1,v_3,\ldots, v_{\ell-4} \}$ and $P_{v_{\ell-2}}^2=\{v_{\ell-2},v_{\ell-3},v_0,v_1,v_2,\ldots, v_{\ell-5},v_{\ell-4} \}$. Then all elements in $c(P_{v_{\ell-2}}^1)$ are even except for $\{1,2\ell-5 \}$ and all elements in $c(P_{v_{\ell-2}}^2)$ are odd except for $\ell-3$. So we have $c(H^*_{v_{\ell-2}})\subseteq \{1,2\ell-5,\ell-3 \}$. Change the positions of a few vertices in $P_{v_{\ell-2}}^2$, we can get $P_{v_{\ell-2}}^{21}=\{v_{\ell-2},v_{\ell-3},v_0,v_2,v_1,v_3,\ldots,$ $v_{\ell-4} \}$ and $P_{v_{\ell-2}}^{22}=\{v_{\ell-2},v_{\ell-4},v_{\ell-3},v_2,v_0,v_1,v_3,\ldots, v_{\ell-5} \}$. There is no $1$ in $c(P_{v_{\ell-2}}^{21})$ and no $\{\ell-3,2\ell-5 \}$ in $c(P_{v_{\ell-2}}^{22})$. Thus $c(H^*_{v_{\ell-2}})=\emptyset$.\par

The proof is similar when $\ell$ is even.
\end{proof}

\begin{cor}\label{Cc1}
The graph $H$ is an $(\mathfrak{R}(P_\ell),t)$-saturated graph for $\ell\geq 7$.
\end{cor}

\begin{proof}
By Lemma \ref{Ct2}, for any $e\in E(\bar{H})$, then $H+e$ contains a rainbow $P_\ell$ such that $c(e)$ is any color from $[t]$ and $e$ is incident to $v_i\in H^*$, where $i\notin \{\ell-4,\ell-3 \}$. By Definition \ref{Cd1}, $c(v_{\ell-2}v_{\ell-3})=2\ell-5=c(xv_{\ell-4})$ and $c(v_{\ell-2}v_{\ell-4})=2\ell-6=c(xv_{\ell-3})$. Then we have $c(H_x)=c(H_{v_{\ell-2}})=\emptyset$. \par
If $\ell\geq 7$, then $H+e$ contains a rainbow $P_\ell$ also holds for $e$ which is incident to $v_{\ell-4},v_{\ell-3}$ by Definition \ref{Cd1}.

Together with Lemma \ref{Ct1}, $H$ is $(\mathfrak{R}(P_\ell),t)$-saturated for $\ell\geq 7$.
\end{proof}

\begin{prop}\label{pro1}
For $\ell=5$, $G^*$ in Definition \ref{Cd1} is $(\mathfrak{R}(P_5),5)$-saturated and $|E(G^*)|=\frac{7n}{5}$.
\end{prop}

\begin{proof}
Similar as the proof in Lemma \ref{Ct2}, $c(H_{v_0})=\emptyset, \ c(H_{v_3})=c(H_x)=\emptyset$.

If $v=v_1$, then there are only $P_{v_1}^1=\{v_1,v_0,v_2,v_3 \}$ and $P_{v_1}^2=\{v_1,v_3,v_2,v_0 \}$ in $H^*$. Then $c(H^*_{v_1})=\{2,5 \}$.

If $v=v_2$, then there are only $P_{v_2}^1=\{v_2,v_0,v_1,v_3 \}$ and $P_{v_2}^2=\{v_2,v_3,v_1,v_0 \}$ in $H^*$. Then $c(H^*_{v_2})=\{1,4 \}$.

For $c(v_2 x)=4, \ c(v_1 x)=5$, we have $c(H_{v_1})=\{2 \}$ and $c(H_{v_2})=\{1 \}$.

For $H^1,H^2 \subset G^*$, let $e_1=v_1^1 v_1^2 \in E(\overline{G^*})$ and $c(e_1)=2$. We can get a rainbow $P_5=\{v_2^1,v_3^1,v_1^1,v_1^2,v_2^2 \}$ in $G^*+e_1$.

Let $e_2=v_2^1 v_2^2 \in E(\overline{G^*})$ and $c(e_2)=1$, we can get a rainbow $P_5=\{v_1^1,v_3^1,v_2^1,v_2^2,v_1^2 \}$ in $G^*+e_2$.

Let $e_3=v_1^1 v_2^2 \in E(\overline{G^*})$ and $c(e_3)\in \{1,2 \}$, we can get a rainbow $P_5=\{v_2^1,v_3^1,v_1^1,v_2^2,v_1^2 \}$ in $G^*+e_3$.

Therefore, $G^*$ is $(\mathfrak{R}(P_5),5)$-saturated and $|E(G^*)|=\frac{7n}{5}$.
\end{proof}

\begin{prop}\label{pro2}
For $\ell=6$, $G^*$ in Definition \ref{Cd1} is $(\mathfrak{R}(P_6),7)$-saturated and $|E(G^*)|=\frac{5n}{3}$.
\end{prop}

\begin{proof}
Similar as in the proof of Lemma \ref{Ct2}, $c(H_{v_0})=\emptyset, \ c(H_{v_1})=\emptyset, \ c(H_{v_4})=c(H_x)=\emptyset$.\par
If $v=v_2$, then there are only $P_{v_2}^1=\{v_2,v_0,v_1,v_3,v_4 \}$, $P_{v_2}^2=\{v_2,v_4,v_3,v_0,v_1 \}$ and $P_{v_2}^3=\{v_2,v_4,v_3,v_1,v_0 \}$ in $H^*$. Then $c(H^*_{v_2})=\{1,7 \}$.

If $v=v_3$, then there are only $P_{v_3}^1=\{v_3,v_1,v_0,v_2,v_4 \}$, $P_{v_3}^2=\{v_3,v_4,v_2,v_1,v_0 \}$ and $P_{v_3}^3=\{v_3,v_4,v_2,v_0,v_1 \}$ in $H^*$. Then $c(H^*_{v_3})=\{1,6 \}$.

For $c(v_2 x)=7, \ c(v_3 x)=6$, we have $c(H_{v_2})=\{1 \}=c(H_{v_3})$.

For $H^1,H^2 \subset G^*$, let $e_1=v_2^1 v_2^1 \in E(\bar{G^*})$ and $c(e_1)=1$. We can get a rainbow $P_6=\{v_0^1,v_3^1,v_4^1,v_2^1,v_2^2,v_0^2 \}$ in $G^*+e_1$.

Let $e_2=v_3^1 v_3^2 \in E(\bar{G^*})$ and $c(e_2)=1$, we can get a rainbow $P_6=\{v_1^1,v_2^1,v_4^1,v_3^1,v_3^2,v_1^2 \}$ in $G^*+e_2$.

Let $e_3=v_2^1 v_3^2 \in E(\bar{G^*})$ and $c(e_3)=1$, we can get a rainbow $P_6=\{v_0^1,v_3^1,v_4^1,v_2^1,v_3^2,v_1^2 \}$ in $G^*+e_3$.

Therefore, $G^*$ is $(\mathfrak{R}(P_6),7)$-saturated and $|E(G^*)|=\frac{5n}{3}$.
\end{proof}

From Corollary \ref{Cc1}, Propositions \ref{pro1} and \ref{pro2}, we get the following theorem.
\begin{thm}
The graph $G^*$ is an $(\mathfrak{R}(P_\ell),t)$-saturated graph and $|E(G^*)|= \lceil \frac{n}{\ell} \rceil \cdot (\binom{\ell-2}{2}+4)$, where $t=2\ell-5$ and $\ell\ge 5$.
\end{thm}

\section{Proof of Theorem \ref{thm1}}

Let $G$ be an $(\mathfrak{R}(P_\ell),t)$-saturated graph with $n$ vertices. We present the proof by considering the following cases.

{\bf Case 1.} $\ell=5$.

If $t=5$ and $n \equiv 0\ (mod \ 5)$, then all components of $G$ are $H$ in Definition \ref{Cd1}, where $|H|=5$, i.e., $G=G^*$, which yields $|E(G)|=\frac{7n}{5}$.

If $t=6$ and $n \equiv 1\ (mod \ 5)$, then all components of $G$ are $H$ except for two rainbow triangles, which yields $|E(G)|=\frac{7n}{5}-\frac{12}{5}$.

If $t=5$ and $n \equiv 2\ (mod \ 5)$, then all components of $G$ are $H$ except for an edge $e$ that $c(e)=3$, which yields $|E(G)|=\frac{7n}{5}-\frac{9}{5}$.

If $t=5$ and $n \equiv 3\ (mod \ 5)$, then all components of $G$ are $H$ except for a rainbow triangle, which yields $|E(G)|=\frac{7n}{5}-\frac{6}{5}$.

If $t=5$ and $n \equiv 4\ (mod \ 5)$, then all components of $G$ are $H$ except for a rainbow $K_4$, which yields $|E(G)|=\frac{7n}{5}+\frac{2}{5}$.

Thus $sat_t(n,\mathfrak{R}(P_5))\le \lceil \frac{7n}{5} \rceil$.\par

{\bf Case 2.} $\ell=6$.

If $t=7$ and $n \equiv 0\ (mod \ 6)$, then all components of $G$ are $H$ in Definition \ref{Cd1}, where $|H|=6$, i.e., $G=G^*$, which yields  $|E(G)|=\frac{5n}{3}$.

If $t=7$ and $n \equiv 1\ (mod \ 6)$, then all components of $G$ are $H$ except for a rainbow triangle and a rainbow $K_4$, which yields  $|E(G)|=\frac{5n}{3}-\frac{8}{3}$.

If $t=7$ and $n \equiv 2\ (mod \ 6)$, then all components of $G$ are $H$ except for an edge $e$ that $c(e)=4$, then $|E(G)|=\frac{5n}{3}-\frac{7}{3}$.

If $t=7$ and $n \equiv 3\ (mod \ 6)$, then all components of $G$ are $H$ except for a rainbow triangle, which yields  $|E(G)|=\frac{5n}{3}-2$.

If $t=7$ and $n \equiv 4\ (mod \ 6)$, then all components of $G$ are $H$ except for a rainbow $K_4$, which yields  $|E(G)|=\frac{5n}{3}-\frac{2}{3}$.

If $t=7$ and $n \equiv 5\ (mod \ 6)$, then all components of $G$ are $H$ except for two rainbow $K_4$ and a rainbow triangle, which yields  $|E(G)|=\frac{5n}{3}-\frac{10}{3}$.

Thus $sat_t(n,\mathfrak{R}(P_6))\leq \lceil \frac{5n}{3} \rceil$.\par

{\bf Case 3.} $\ell\geq 7$.

If $t=2\ell-5$ and $n \equiv 0\ (mod \ \ell)$, then all components of $G$ are $H$ in Definition \ref{Cd1}, where $|H|=\ell$, i.e., $G=G^*$, which yields $|E(G)|=\frac{n}{\ell}\cdot (\binom{\ell-2}{2}+4)$.

If $t=2\ell-5$ and $n \equiv 1\ (mod \ \ell)$, then  all components of $G$ are $H$ except for a vertex $v$, which yields $|E(G)|=\frac{n-1}{\ell}\cdot (\binom{\ell-2}{2}+4)$.

If $t=2\ell-5$ and $n \equiv 2\ (mod \ \ell)$, then  all components of $G$ are $H$ except for an edge $e$, where $c(e)\in [t]$, which yields $|E(G)|=\frac{n-2}{\ell}\cdot (\binom{\ell-2}{2}+4)+1$.

If $t=2\ell-5$ and $n \equiv 3\ (mod \ \ell)$, then  all components of $G$ are $H$ except for a rainbow triangle, which yields $|E(G)|=\frac{n-3}{\ell}\cdot (\binom{\ell-2}{2}+4)+3$.

If $t=2\ell-5$ and $n \equiv 4\ (mod \ \ell)$, then all components of $G$ are $H$ except for a rainbow $K_4$, which yields $|E(G)|=\frac{n-4}{\ell}\cdot (\binom{\ell-2}{2}+4)+6$.

$\ldots$

If $t=2\ell-5$ and $n \equiv a\ (mod \ \ell)$, then  all components of $G$ are $H$ except for a rainbow $K_a$, which yields $|E(G)|=\frac{n-a}{\ell}\cdot (\binom{\ell-2}{2}+4)+\binom{a}{2}$. It is easy to check $e_a\le \lceil \frac{n}{\ell} \rceil \cdot (\binom{\ell-2}{2}+4) $,   since $1\le a \le \ell-1$ and then $1\le a \le \sqrt{(\ell-3)^2+8}$.

From all the above cases, we have for $\ell\ge 5$ and $t\ge 2\ell-5$, $sat_t(n,\mathfrak{R}(P_\ell))\le \lceil \frac{n}{\ell} \rceil \cdot \left(\binom{\ell-2}{2}+4\right)$. The proof is thus complete. \qed\\


\noindent{\bf Acknowledgements.} Shujuan Cao is partially supported by the National Natural Science Foundation of China (No. 11801412 ).


\begin{thebibliography}{111}

\bibitem{BF}
M. D. Barrus, M. Ferrara, J. Vandenbussche and P. S. Wenger, Colored saturation parameters for rainbow subgraphs,
\emph{J. Graph Theory}, {\bf 86} (2017), 375--386.

\bibitem{EH}
P. Erd{\"o}s, A. Hajnal and J.W. Moon, A problem in graph theory,
\emph{Amer. Math. Monthly}, {\bf 71} (1964), 1107--1110.

\bibitem{FFS2011} J. Faudree, R. Faudree and J. Schmitt, A survey of minimum saturated graphs, {\it Electron. J. Combin.} {\bf18} (2011), \#DS19.

\bibitem{FJ}
M. Ferrara, D. Johnston, S. Loeb,
F. Pfender, A. Schulte, H.C. Smith,
E. Sullivan, M. Tait, C. Tompkins, On edge-colored saturation problems, arXiv:1712.00163.


\bibitem{GI}
A. Gir\~{a}o (2019). {\it Extremal and Structural Problems of Graphs} (Doctoral thesis). https://doi.org/10.17863/CAM.32787

\bibitem{AG}
A. Gir\~{a}o, D. Lewis and K. Popielarz, Rainbow saturation of graphs, arXiv:1710.08025.

\bibitem{DK}
D. Kor{\'a}ndi, Rainbow saturation and graph capacities,
\emph{SIAM J. Discrete Math.} {\bf 32}(2)(2018), 1261--1264.

\bibitem{W}
D.B. West, Introduction to Graph Theory, Prentice Hall, Inc., Upper Saddle River, NJ, 2001.

\bibitem{Z}
A. Zykov, On some properties of linear complexes (in Russian),
\emph{Mat. Sbornik N. S.}, {\bf 24} (1949), 163--188.

\end{thebibliography}
\end{document}